\documentclass[12pt]{amsart}
 \usepackage{amscd,amsmath,amsthm,amssymb}
 \usepackage[normalem]{ulem}
 \usepackage[dvips]{graphicx}
 \usepackage{pstcol,pst-plot,pst-3d}
\newpsstyle{fatline}{linewidth=1.5pt}
\newpsstyle{fyp}{fillstyle=solid,fillcolor=verylight}
\definecolor{verylight}{gray}{0.97}
\definecolor{light}{gray}{0.93}
\definecolor{medium}{gray}{0.82}
 \def\NZQ{\Bbb}               
 \def\NN{{\NZQ N}}

 \def\frk{\frak}               

 \def\mm{{\frk m}}

 \def\opn#1#2{\def#1{\operatorname{#2}}} 
 \opn\chara{char} \opn\length{\ell} \opn\pd{pd} \opn\rk{rk}
 \opn\projdim{proj\,dim} \opn\injdim{inj\,dim} \opn\rank{rank}
 \opn\depth{depth} \opn\grade{grade} \opn\height{height}
 \opn\embdim{emb\,dim} \opn\codim{codim}
 
 \opn\Tr{Tr} \opn\bigrank{big\,rank}
 \opn\superheight{superheight}\opn\lcm{lcm}
 \opn\trdeg{tr\,deg}
 \opn\reg{reg} \opn\lreg{lreg} \opn\ini{in} \opn\lpd{lpd}
 \opn\size{size} \opn\sdepth{sdepth}
 \opn\link{link}\opn\fdepth{fdepth}\opn\lex{lex}

 \opn\div{div} \opn\Div{Div} \opn\cl{cl} \opn\Cl{Cl}
 \opn\Spec{Spec} \opn\Supp{Supp} \opn\supp{supp} \opn\Sing{Sing}
 \opn\Ass{Ass} \opn\Min{Min}\opn\Mon{Mon}
 \opn\Ann{Ann} \opn\Rad{Rad} \opn\Soc{Soc}
 \opn\Im{Im} \opn\Ker{Ker} \opn\Coker{Coker} \opn\Am{Am}
 \opn\Hom{Hom} \opn\Tor{Tor} \opn\Ext{Ext} \opn\End{End}
 \opn\Aut{Aut} \opn\id{id}
 
 \opn\nat{nat}
 \opn\pff{pf}
 \opn\Pf{Pf} \opn\GL{GL} \opn\SL{SL} \opn\mod{mod} \opn\ord{ord}
 \opn\Gin{Gin} \opn\Hilb{Hilb}\opn\sort{sort}

 \opn\aff{aff} \opn
 \con{conv} \opn\relint{relint} \opn\st{st}
 \opn\lk{lk} \opn\cn{cn} \opn\core{core} \opn\vol{vol}
 \opn\link{link} \opn\star{star}\opn\lex{lex}\opn\set{set}\opn\purelex{purelex}
 \opn\dist{dist}
 \opn\gr{gr}
 
 \def\pot#1#2{#1[\kern-0.28ex[#2]\kern-0.28ex]}

 \opn\dirlim{\underrightarrow{\lim}}
 \opn\inivlim{\underleftarrow{\lim}}
 \def\Implies{\ifmmode\Longrightarrow \else
         \unskip${}\Longrightarrow{}$\ignorespaces\fi}
 \def\implies{\ifmmode\Rightarrow \else
         \unskip${}\Rightarrow{}$\ignorespaces\fi}
 \def\iff{\ifmmode\Longleftrightarrow \else
         \unskip${}\Longleftrightarrow{}$\ignorespaces\fi}

 \let\:=\colon
 \newtheorem{Theorem}{Theorem}[section]
 \newtheorem{Lemma}[Theorem]{Lemma}
 \newtheorem{Corollary}[Theorem]{Corollary}
 \newtheorem{Proposition}[Theorem]{Proposition}
 \newtheorem{Remark}[Theorem]{Remark}
 
 \newtheorem{Example}[Theorem]{Example}
 
 \newtheorem{Definition}[Theorem]{Definition}

 \let\epsilon\varepsilon
 \let\kappa=\varkappa
 \def\qed{\ifhmode\textqed\fi
       \ifmmode\ifinner\quad\qedsymbol\else\dispqed\fi\fi}
 \def\textqed{\unskip\nobreak\penalty50
        \hskip2em\hbox{}\nobreak\hfil\qedsymbol
        \parfillskip=0pt \finalhyphendemerits=0}
 \def\dispqed{\rlap{\qquad\qedsymbol}}
 
 \opn\dis{dis}
 \def\pnt{{\raise0.5mm\hbox{\large\bf.}}}
 
 \opn\Lex{Lex}

\begin{document}

\title{ Ideals with linear quotients and componentwise polymatroidal ideals}

\author{Somayeh Bandari}
\address{Imam Khomeini International University-Buein Zahra Higher Education Center of Engineering and Technology, Qazvin, Iran}
\email{somayeh.bandari@yahoo.com and s.bandari@bzte.ac.ir}

\author{ Ayesha Asloob Qureshi}
\address{Sabanci University, Faculty of Engineering and Natural Sciences, Orta Mahalle, Tuzla 34956, Istanbul,
Turkey}
\email{aqureshi@sabanciuniv.edu}



\maketitle

\begin{abstract}
If $I$ is a monomial ideal with linear quotients, then it has
 componentwise linear quotients.
However, the converse of this statement is an open question. In this
paper, we provide two classes of ideals for which the converse of this
statement holds. First class is the componentwise polymatroidal
ideals in $K[x,y]$ and the second one is the componentwise
polymatroidal ideals with strong exchange property.
\end{abstract}

\section*{Introduction}

Let $S=K[x_1,\ldots,x_n]$ be the polynomial
ring in $n$ indeterminates over an arbitrary field $K$. For a monomial ideal $I\subset S$,
 we denote the unique minimal set of monomial generators of  $I$ by $G(I)$. For a monomial $u=x_1^{a_1}\cdots x_n^{a_n}\in S$,
  we set $\deg_{x_i}(u):=a_i$ for all $i=1, \ldots, n$.
 Let $I$ be generated in a single degree. Then $I$ is said to be {\em polymatroidal},
 if for any $u,v\in G(I)$ with $\deg_{x_i}(u)> \deg_{x_i}(v)$,
there exists an index $j$ with $\deg_{x_j}(u)< \deg_{x_j}(v)$ such
that $x_j(u/x_i)\in I$. Equivalently,  a monomial ideal $I$ in $S$ is called a
 {\em polymatroidal ideal}, if there exists a set of bases $\mathcal{B}\subset\mathbb{Z}^n$ of a polymatroid of rank $d$, such that
$\mathcal{G}(I)=\{{\bf x}^{\bf a} : {\bf a} \in \mathcal{B}\}$. The polymatroidal ideals were introduced as a
 non-squarefree analogue of matroidal ideals, and they have been discussed in various articles with both algebraic and
 combinatorial points of view. An overview of discrete polymatroids
 and polymatroidal ideals is provided in \cite[Chapter 12]{HHbook}. The most distinguished properties of polymatroidal ideals are: product of polymatroidal
 ideals is again polymatroidal, and polymatroidal ideals have linear resolutions. The definition of polymatroidal ideals was
 generalized in \cite{BH} by removing the restriction on the degrees of generators. Let $I\subset S$ be a monomial ideal.
 Then $I$ is called a {\em componentwise polymatroidal} if $I_{\langle j\rangle}$ is
polymatroidal for all $j$, where $I_{\langle j\rangle}$ is the
$j$-th graded component of $I$.

A monomial ideal $I\subset S$ is said to have {\em linear
quotients}, if there exists an ordering $u_1,\ldots,u_m$ of its
minimal generators such that for each $i=2,\ldots,m$, the colon
ideal $(u_1,\ldots,u_{i-1}):u_i$ is generated by a subset of the
variables. Any such ordering of $G(I)$ for which $I$ has linear
quotients is called an {\em admissible order}. The ideals with
linear quotients were introduced by Herzog and Takayama in
\cite{HT}. In  \cite[Theorem 2.7]{SZ}, Soleyman Jahan and Zheng
proved that if a monomial ideal $I$ has linear quotients, then it
has componentwise linear quotients, which means that each graded
component of $I$ has linear quotients. They raised the converse of this statement as an open question.\\
{\bf Question 1:} If a monomial ideal has componentwise linear quotients, then does it admit linear quotients?\\
 The polymatroidal ideals have linear quotients, see \cite[Lemma
1.3]{HT}, and hence
 componentwise polymatroidal ideals admit componentwise linear quotients.
 Question 1 still remains open even for the class of  componentwise polymatroidal ideals.
 First author and Herzog in \cite[Corollary 3.7]{BH} proved that the
componentwise of Veronese type ideals have linear quotients, which
is a class of componentwise polymatroidal ideals. The admissible
order introduced in \cite{BH} is very special and cannot be
generalized to arbitrary componentwise polymatroidal ideals. In this
paper, our motivation is to continue the work done in \cite{BH}.
 We give an affirmative answer to Question 1, for two different classes of componentwise polymatroidal ideals.

A breakdown of the contents of this paper is as follows: in Section~\ref{Sec1},
we discuss certain exchange properties for monomial ideals and in Proposition~\ref{exchange}, we characterize
componentwise polymatroidal ideals by slightly generalizing the notion of non-pure exchange property.
The ideals with  non-pure exchange property were introduced in \cite{BH}. We introduce the ideals
 with non-pure dual exchange property, see Definition~\ref{nonpuredual}. In Proposition~\ref{com},
 we show that the componentwise polymatroidal ideals satisfy non-pure dual exchange property. Moreover, monomial ideals
generated in a single degree with non-pure dual exchange property
are exactly the polymatroidal ideals, see Proposition \ref{non}.

 In Section~\ref{Sec2}, we consider componentwise polymatroidal ideals generated in two variables.
  In Proposition \ref{equal}, we prove that the concept of non-pure dual exchange property,
    non-pure exchange property, and componentwise polymatroidal coincide for monomial ideals in $K[x,y]$.
  Such a statement does not hold if we increase the number of variables to three. The first main
  result of Section \ref{Sec2} is Corollary~\ref{cortight}, in which we give a complete description of
   componentwise polymatroidal ideals in $K[x,y]$. For this we recall the notion of $x$-tight and $y$-tight
   ideals from \cite{Qu}. Our next main result
  is Theorem~\ref{2quotient} which states that the componentwise polymatroidal ideals in $K[x,y]$ have linear quotients.
   The admissible order introduced in Theorem~\ref{2quotient} cannot be extended for the ideals generated in three variables.
   In general, the powers of componentwise polymatroidal ideals need not to be componentwise polymatroidal, as noted in \cite{BH}.
   However, with the help of Corollary~\ref{cortight}, we establish that the product of componentwise polymatroidal ideals in
    $K[x,y]$ is again componentwise polymatroidal ideal, see Corollary \ref{corpro}.

In  Section~\ref{Sec3}, we give another class of componentwise
polymatroidal ideals with linear quotients. Indeed, we prove that
the componentwise polymatroidal ideals with strong exchange property
have linear quotients, see Theorem \ref{main}.

\section{Componentwise polymatroidal ideals and certain exchange properties}\label{Sec1}

Let $I\subset S=K[x_1,\ldots,x_n]$ be a monomial ideal generated in
a single degree. Then $I$ is said to be {\em polymatroidal}, if for
any two monomials $u,v\in G(I)$ with $\deg_{x_i}(u)> \deg_{x_i}(v)$,
there exists an index $j$ with $\deg_{x_j}(u)< \deg_{x_j}(v)$ such
that $x_j(u/x_i)\in I$. In other words, the generators of
polymatroidal ideals satisfy the so called ``exchange property".
 Let $I\subset S$ be a monomial ideal.
 Then $I$ is called a {\em componentwise polymatroidal} if $I_{\langle j\rangle}$ is polymatroidal for all $j$, where $I_{\langle j \rangle}$ is the
$j$-th graded component of $I$. It is natural to ask whether
componentwise polymatroidal can also be characterized in terms of a
 certain exchange property. In \cite{BH}, the term of ideals with non-pure exchange property was introduced as follows:

\begin{Definition}\cite[Definition 3.4]{BH}
Let $I\subset S$ be a monomial ideal. Then $I$ is said to satisfy
{\em non-pure exchange property} if, for all $u,v \in G(I)$
  with $\deg(u)\leq \deg(v)$ and for all $i$ such that $\deg_{x_i}(v) > \deg_{x_i} (u)$, there exists $j$ such that
   $\deg_{x_j}(v)< \deg_{x_j}(u)$ and $x_j(v/x_i) \in I$.
 \end{Definition}

From \cite[Proposition 3.5]{BH}, it follows that componentwise
polymatroidal ideals satisfy non-pure exchange property.
 However, it is also noted in \cite[page 762]{BH} that ideals with non-pure exchange property are not necessarily componentwise polymatroidal.
  The following proposition shows that the componentwise polymatroidal ideals can be characterized in terms of a certain exchange property
   by slightly modifying the definition of non-pure exchange property.

\begin{Proposition}\label{exchange}
Let $I \subset S=K[x_1, \ldots, x_n]$ be a monomial ideal. Then the
following statements are equivalent:
\begin{enumerate}
\item $I$ is componentwise polymatroidal.
\item  If $u,v \in I$ with $\deg(u)\leq \deg(v)$ and $u \nmid v$, then for all $i$
with $\deg_{x_i}(v) > \deg_{x_i} (u)$, there exists $j$ such
that $\deg_{x_j}(v)< \deg_{x_j}(u)$ and $x_j(v/x_i) \in I$.
\end{enumerate}
\end{Proposition}
\begin{proof}
(2) $\implies$ (1): It follows  immediately from the definition of
componentwise polymatroidal ideals.

(1) $\implies$ (2): Let $I$ be a componentwise polymatroidal, and
$u,v \in I$ with $\deg(u)\leq \deg(v)$ and
 $u \nmid v$. Let $\deg_{x_i}(v) > \deg_{x_i} (u)$ for some $1 \leq i \leq n$. We need to show that there exists $j$
 such that $\deg_{x_j}(v)< \deg_{x_j}(u)$ and $x_j(v/x_i) \in I$.
Let $r=\deg(u)$ and $ \deg(v)=t$. If $r=t$, then $u,v \in I_{\langle
r \rangle}$, and the assertion
 follows directly from the definition of componentwise polymatroidal ideals.
 Now, let $r<t$. Since $u \nmid v$, it follows that  there exists some $k$ such that
   $\deg_{x_k}(v)< \deg_{x_k}(u)$. Set $w:=x_k^{t-r}u$. This gives $w,v \in I_{\langle t \rangle}$,
   and $\deg_{x_i}(v) > \deg_{x_i} (w)=\deg_{x_i} (u)$.  Again, by following the definition of componentwise polymatroidal ideal,
    we obtain some $j$ such that $\deg_{x_j}(v)< \deg_{x_j}(w)$ and $x_j(v/x_i) \in I$, as required.
\end{proof}

In \cite[Lemma 3.1]{HH} authors proved that there exists a ``dual
version" of exchange property for polymatroidal ideals.
 More precisely, they proved that if $I$ is a polymatroidal ideal, then for any monomials $u = x_1^{a_1}\cdots x_n^{a_n}$
 and $v = x^{b_1}\cdots x_n^{b_n} $ in $G(I)$ and for each $i$ with $a_i <b_i,$ one has $j$ with $a_j>b_j$ such that $x_iu/x_j \in G(I)$. We
remark here that such a dual version does not hold for ideals with
non-pure exchange property.
  For example, take $I=(x_1x_3x_4, x_1x_4x_5,x_1x_2x_4,x_1x_2^2x_6,x_3x_4^2x_5^2)$
  and $u=x_3x_4^2x_5^2$ and $v=x_1x_2^2x_6$. Then $\deg_{x_6} (u) < \deg_{x_6} (v)$, but $x_6
u/x_i \notin I$ for any $i=3,4,5$. Motivated by this, we define the
following:

\begin{Definition}\label{nonpuredual}
Let $I\subset S$ be a monomial ideal. We say that $I$ satisfies {\em
non-pure dual exchange property} if, for all $u,v \in G(I)$ with
$\deg(u) \leq \deg(v)$ and for all $i$ such that $\deg_{x_i} (v) <
\deg_{x_i} (u)$, there exists $j$ such that $\deg_{x_j} (v) >
\deg_{x_j} (u)$ and $x_i(v/x_j) \in I$.
\end{Definition}

Note that non-pure dual exchange property and non-pure exchange
property do not imply each other. For example, $I=(x_1x_3x_4,
x_1x_4x_5,x_1x_2x_4, x_1x_2^2x_6,x_3x_4^2x_5^2)$ satisfies non-pure
exchange property,
 but it does not satisfy non-pure dual exchange property.
 On the other hand, let
$I=(x_1x_2,x_1x_4^2,x_2x_4^2,x_3x_4^2)$. Then $I$ satisfies non-pure
dual exchange property, but it does not satisfy non-pure exchange
property. The ideals with non-pure dual exchange property behave
slightly better
 than the ideals with non-pure exchange property. We begin by first proving the following:

\begin{Proposition}\label{non}
Let $I\subset S$ be a monomial ideal generated in a single degree.
Then $I$ is polymatroidal if and only if $I$ has  non-pure dual
exchange property.
\end{Proposition}

\begin{proof}
The ``Only if" part follows from \cite[Lemma 3.1]{HH}.

To prove the converse, assume that $I$ is a monomial ideal with
non-pure dual exchange property.  We introduce the distance of
$u=x_1^{a_1}\cdots x_n^{a_n}\in G(I)$ and $v=x_1^{b_1}\cdots
x_n^{b_n}\in G(I)$  by setting $\dist(u,v):=1/2\sum_{q=1}^n
|a_q-b_q|$. Fix $i$ with $a_i<b_i$. If there exists $k_1\neq i$ with
$a_{k_1}<b_{k_1}$, then there exists $l_1$ with $a_{l_1}>b_{l_1}$
such that $w_1=x_{k_1}u/x_{l_1}\in G(I)$. Let $w_1=x_1^{c_1}\cdots
x_n^{c_n}$. Then $c_i=a_i$ and $\dist(w_1,v)<\dist(u,v)$. Again, if
there exists $k_2\neq i$ with $c_{k_2}<b_{k_2}$, then there exists
$l_2$ with $c_{l_2}>b_{l_2}$  such that $w_2=x_{k_2}w_1/x_{l_2}\in
G(I)$. Let $w_2=x_1^{d_1}\cdots x_n^{d_n}$. Then $d_i=a_i$ and
$\dist(w_2,v)<\dist(w_1,v)$. Repeating this procedure yields
$w=x_1^{q_1}\cdots x_n^{q_n}\in G(I)$ with $q_i=a_i<b_i$ and
$q_j\geq b_j$ for all $j\neq i$. One has $j_0\neq i$ with
$a_{j_0}=q_{j_0}>b_{j_0}$. Then $x_{j_0}v/x_i\in G(I)$, as desired.
\end{proof}

In case of componentwise polymatroidal ideals, we have the following:

\begin{Proposition} \label{com}
The componentwise polymatroidal ideals satisfy the non-pure dual
exchange property.
\end{Proposition}

\begin{proof}
Let $I$ be a componentwise polymatroidal ideal. Then by  \cite[Lemma
3.1]{HH}, $I$ satisfies non-pure dual exchange property for each
$I_{\langle a \rangle}$. Let $u,v \in G(I)$. Then the assertion is
clear when $\deg(u) = \deg(v)$.

Let $\deg(u) < \deg(v)=t$ and $i$ be such that $\deg_{x_i} (v) <
\deg_{x_i} (u)$. Let $b$ be an integer
 such that $\deg (v)=\deg(ux_i^b)$. Note that $\deg_{x_i} (v)< \deg_{x_i}(ux_i^b)$ and $v, ux_i^b \in I_{\langle t \rangle}$.
  Then again by using  \cite[Lemma 3.1]{HH} we see that $x_i(v/x_j) \in I$, for some $j$ with $\deg_{x_j} (v) >\deg_{x_j} (ux_i^b)$.
  Moreover, $\deg_{x_j} (ux_i^b)=\deg_{x_j} (u)$, for all $j \neq i$, so we conclude that
   $x_i (v/x_j) \in I$, for some $j$ with $\deg_{x_j} (v) >\deg_{x_j} (u)$, as required.
\end{proof}

It is natural to ask whether the converse of Proposition \ref{com}
holds true. In the following example, we give a negative answer to
this.
\begin{Example}
Let $I=(x_1x_2,x_1x_4^2,x_2x_4^2,x_3x_4^2)$. Then $I$ has non-pure
dual exchange property, but $I_{\langle 3 \rangle}$ is not polymatroidal. Indeed, we have
$x_1^2x_2\in G(I_{\langle 3 \rangle})$ and $x_3x_4^2\in G(I_{\langle 3 \rangle})$, but
$x_1x_3x_4\notin G(I_{\langle 3 \rangle})$ and  $x_2x_3x_4\notin G(I_{\langle 3 \rangle})$. In
particular, it follows by Proposition \ref{non} that $I_{\langle 3 \rangle}$ does
not have non-pure dual exchange property.
\end{Example}

We summarize  Section~\ref{Sec1} as follows. The non-pure exchange
property and the non-pure dual exchange property do not imply each
other, and an ideal admitting any of these properties need not to be
componentwise polymatroidal. On the other hand, if $I$ is a monomial
ideal generated in a single degree, then $I$ admits non-pure
exchange property if and only if it admits non-pure dual exchange
property. Furthermore, in this case, these properties are equivalent
to the definition of polymatroidal ideals.

\section{Componentwise polymatroidal ideals in  $K[x,y]$}\label{Sec2}

In this section, we discuss the componentwise polymatroidal ideals
in $K[x,y]$. Our main goal is to show that componentwise
polymatroidal ideals in $K[x,y]$ preserve
 some nice properties of
polymatroidal ideals. More precisely, we will show that they have
linear quotients, and their powers are again componentwise
polymatroidal.

Given a monomial ideal $I \subset S=K[x,y]$,
we will always assume that the elements of $G(I)$ are arranged with respect to pure lexicographical order induced by $x>y$.
\begin{enumerate}
\item[($*$)] Let $G(I)=\{u_0, \ldots, u_m\}$ with $u_i=x^{a_i}y^{b_i}$,
for all $i=0, \ldots, m$. Then we have
\[
a_0>a_1>\cdots > a_m {\text{ and }} b_0<b_1 < \cdots < b_m
\]
In particular, if $i<j$ then $a_i> a_j$ and $b_i< b_j$.
\end{enumerate}

Next, we recall the definition of $x$-tight and $y$-tight ideal from \cite{Qu}.
 In addition, we give the definition of $yx$-tight ideal to facilitate the description of componentwise polymatroidal ideals.
  Given $a,b \in \NN$ with $a<b$, we set $[a,b]:=\{c \in \NN: a\leq c \leq b\}$.

  \begin{Definition}\label{tightdef}
Let $\mm=(x,y) \subset K[x,y]$ and $I \subset K[x,y]$ be an
$\mm$-primary monomial ideal with an ordering of $G(I)$ as in $(*)$, that is, $a_m=b_0=0$. Then
\begin{enumerate}
\item $I$ is called {\em $x$-tight} if $a_{m-i}=i$, for all $i=0, \ldots, m$. To simplify the notation, if needed, we will refer to $I$ as $x$-tight in $[0,m]$.
\item $I$ is called {\em $y$-tight} if $b_{i}=i$, for all $i=0, \ldots, m$. To simplify the notation, if needed, we will refer to $I$ as $y$-tight in $[0,m]$.

\item $I$ is called {\em $yx$-tight} if there exists some $0\leq j \leq m$ such that  $b_i=i$,
for all $i=0, \ldots, j$ and $a_{m-i}=i$, for all $i=0,
\ldots, m-j$. In particular, $x$-tight and $y$-tight ideals are also $yx$-tight ideals.  The ideals that are $yx$-tight but not  $y$-tight or $x$-tight, will be referred to as {\em strict $yx$-tight.}
\end{enumerate}
The $x$-tight and $y$-tight ideals are also known as $\mm$-primary lexsegment ideals, see \cite[Remark 2.10]{CNJR}.
\end{Definition}

\begin{Remark}\label{imp}{\em
We emphasize a notable aspect of the definition of strict $yx$-tight
ideals, which plays a crucial role in the next results. If $I$ is a
strict $yx$-tight ideal, then the generating set of  $yx$-tight
ideal is such that its initial part is $y$-tight which is followed
by an $x$-tight part. Also, the monomial $u_j$ where this shift
happens, can be included in both $y$-tight and $x$-tight part of the
ideal. In particular, we have $a_j=m-j$ and $b_j=j$. Therefore, we
can write $I=x^{m-j}J+y^{j}K$ such that $J$ is a $y$-tight ideal
with $G(x^{m-j}J)=\{u_0, \ldots, u_j\}$ and $K$ is an $x$-tight
ideal with $G(y^jK)=\{u_{j}, \ldots, u_{m}\}$. Note that such a
monomial $u_j$ need not to be unique, but it must exist.}
\end{Remark}

\begin{Example}{\em
Let $I_1, I_2$ and  $I_3$ be monomial ideals in $K[x,y]$, with

\[
G(I_1)= \{x^7, x^6y^5, x^5y^8, x^4y^{10}, x^3y^{13}, x^2y^{17},
xy^{20},y^{25}\},
\]

\[
G(I_2)= \{x^{17}, x^{14}y, x^{11}y^2, x^8y^{3}, x^7y^{4}, x^5y^{5},
x^4y^{6},x^2y^{7}, y^8\},
\]
 and
\[
G(I_3)= \{x^{20}, x^{14}y, x^{11}y^2, x^8y^{3}, x^6y^{4}, x^5y^{9},
x^4y^{13},x^3y^{14},x^2y^{16}, xy^{17}, y^{21}\}.
\]
Here $I_1$ is $x$-tight in $[0,7]$ and $I_2 $ is $y$-tight in $[0,8]$. Moreover, $I_3$ is $yx$-tight.
 The monomial $ x^6y^{4}$ is a generator in $I_3$ where $y$-tight part of $G(I_3)$ shifts to $x$-tight.  We have $I_3=x^6J+y^4K$ with
\[
G(x^6J)=\{x^{20}, x^{14}y, x^{11}y^2, x^8y^{3}, x^6y^{4}\} \text{ and $J$ is $y$-tight in $[0,4]$,}
\]
 and
 \[G(y^4K)=\{x^6y^{4}, x^5y^{9}, x^4y^{13},x^3y^{14},x^2y^{16}, xy^{17},y^{21}\}  \text{ and $K$ is $x$-tight  in $[0,6]$.}
 \]
}
\end{Example}

The following proposition shows that the polymatroidal ideals in $K[x,y]$
are exactly those monomial ideals that are
 $x$-tight and $y$-tight at the same time.

\begin{Proposition}\label{polydeg2}
Let $I\subset K[x,y]$ be a polymatroidal ideal. Then $a_i=a_1-(i-1)$
and $b_i=b_1+(i-1)$, for all $i=0, \ldots, m$.
\end{Proposition}

 \begin{proof}
 Pick any $0 \leq i \leq m-1$.  Then $a_i=\deg_x(u_i)> \deg_x(u_{i+1})=a_{i+1}$ and $b_i=\deg_y(u_i)< \deg_y(u_{i+1})=b_{i+1}$.
 The assumption that $I$ is polymatroidal provides us with $y(u_i/x)= x^{a_i-1}y^{b_i+1}\in G(I)$.
  This forces $a_{i+1}=a_i-1$ and $b_{i+1}=b_i+1$. Then the assertion follows immediately by computing values of $a_i$ and $b_i$ recursively.
 \end{proof}

Our next aim is to characterize componentwise polymatroidal in terms
of $yx$-tight ideals. Before this, we prove that the definitions of
ideal with non-pure exchange property, ideal with non-pure dual exchange property and
componentwise polymatroidal ideal coincide
for monomial ideals in $K[x,y]$.  The following result helps us to
switch among these definitions.

 \begin{Proposition}\label{equal}
 Let $I\subset K[x,y]$ be a monomial ideal. Then the following statements are equivalent:
 \begin{enumerate}
 \item $I$ satisfies non-pure dual exchange property.
 \item $I$ satisfies non-pure exchange property.
 \item $I$ is componentwise polymatroidal ideal.
 \end{enumerate}
 \end{Proposition}
  \begin{proof}
 (1) $\iff$ (2) It is a direct consequence of definitions non-pure dual exchange property and non-pure exchange property. \\
 (3) $\Rightarrow$ (2) follows from \cite[Proposition 3.5]{BH}.\\
 (2) $\implies$ (3) Let $d=\min\{ \deg(u): u \in G(I)\}$. We need to show that $I_{<a>}$ is polymatroidal, for all $a \geq d$.
 We will prove this by applying induction on $a$. It follows from the definition of
  non-pure exchange property and the fact that $I_{<d>}$ is  generated in the same degree, that $I_{<d>}$ is a polymatroidal ideal. Assume that $I_{<k>}$
 is polymatroidal for  $a\geq k \geq d$. We want to show that $I_{<a+1>}$ is a
 polymatroidal ideal. Let $u,v\in G(I_{<a+1>})$ with $\deg_{x}(u)<\deg_{x}(v)$ and $\deg_y(u)>\deg_y(v)$. We will show that $(v/x)y\in I$.
 Let $G(I_{<a+1>})=G(\mm I_{<a>})+(w_1,\ldots,w_t)$, where $w_i\in
 G(I)$ and $\deg(w_i)=a+1$ for all $i$. Set $A:=\{w_1,\ldots,w_t\}$. The inductive hypothesis
 implies that $I_{<a>}$ is polymatroidal. It is known that the product of two polymatroidal
 ideals is again polymatroidal, see \cite[Theorem 12.6.3]{HHbook}, therefore, $\mm I_{<a>}$
 is polymatroidal as well. We divide the proof into the following three cases:\\

 {\bf Case 1:} Let $u,v\in A$. Then $(v/x)y\in I$, because $I$ has non-pure  exchange property.\\

{\bf  Case 2:} Let $u\in\mm I_{<a>}$ and $v\in A$. Then
there exists $w\in G(I)$ such that $w\mid u$ and
$\deg(w)<\deg(u)=\deg(v)$.
 Since $\deg_x(w) \leq \deg_x(u)<\deg_x(v)$ and $w,v \in G(I)$, we obtain $\deg_y(w)>\deg_y(v)$. This gives $(v/x)y\in I$
 by following the definition of non-pure  exchange property.\\

 {\bf Case 3:} Let $u\in A$ and $v\in\mm I_{<a>}$. Then there exists $w\in G(I)$ such that $w\mid v$
 and $\deg(w)<\deg(v)=\deg(u)$. Since $\deg_y(w)\leq \deg_y(v)<\deg_y(u)$
 and $w,u \in G(I)$, we obtain $\deg_x(w)>\deg_x(u)$ and  hence $u_1:=(u/y)x\in
 I$. If $u_1=v$, then $(v/x)y=u\in I$, as desired. If $u_1\neq
 v$ and $u_1\in\mm I_{<a>}$. Then, by using the assumption that $\mm I_{<a>}$ is polymatroidal together
 with $\deg_x(u_1)<\deg_x(v)$ and $\deg_y(u_1)>\deg_y(v)$,
 we obtain $(v/x)y\in I$. If $u_1\neq
 v$ and $u_1\in A$. We again argue as before. Since $\deg(w)<\deg(u_1)$ and $\deg_{y}(w)\leq
 \deg_y(v)<\deg_y(u)-1=\deg_y(u_1)$, we have that
 $\deg_x(w)>\deg_x(u_1)$ and $u_2:=(u_1/y)x=(u/y^2)x^2\in
 I$. If $u_2=v$, then $(v/x)y=(u/y)x=u_1\in I$, as desired. Otherwise, we apply the previous argument by
 interchanging the role of $u_1$ to $u_2$. The argument concludes affirmatively after a finite number of
 steps
 because $G(I_{\langle a+1\rangle})$ is a finite set.
 \end{proof}
The following theorem gives a complete characterization of
componentwise polymatroidal ideals in $K[x,y]$.

\begin{Theorem}\label{prop}
Let $I \subset S=K[x,y]$ be a monomial ideal with $G(I)$ ordered as
in $(*)$. If $I$ is componentwise polymatroidal, then the following
statements hold:
\begin{enumerate}
\item[(i)] If $\deg(u_i)> \deg(u_{i+1})$, then  $a_i -1 >a_{i+1}$ and $b_{i+1}=b_i+1$.
\item[(ii)]  If $\deg(u_{i})< \deg(u_{i+1})$,  then $b_i  +1<b_{i+1}$ and $a_i-1=a_{i+1}$.
\item[(iii)] If $\deg(u_{i+1})= \deg(u_i)$,  then $b_i  +1=b_{i+1}$ and $a_{i}-1=a_{i+1}$.
\item[ (iv)] Let $d_i= \deg (u_i)$ for all $i=0,\ldots,m$. Then there exists some $0 \leq j\leq m$ such
that

\[
d_0 \geq d_2 \geq \cdots \geq d_j \leq d_{j+1} \leq \cdots \leq d_m.
\]
\end{enumerate}
Moreover, if an ideal $I$ satisfies above statements, then $I$ is componentwise polymatroidal ideal.
\end{Theorem}

\begin{proof}
Let $I$ be a componentwise polymatroidal ideal. It follows from Proposition~\ref{equal}
 that $I$ has non-pure dual exchange property. We will use this fact repeatedly in the following arguments.\\
(i) Let $\deg(u_i)> \deg(u_{i+1})$. Note that there is no element
$x^cy^d$ in $G(I)$ with
 $a_i>c>a_{i+1}$ and $b_i < d < b_{i+1}$. Moreover, we have $a_i=\deg_x(u_i)> \deg_x(u_{i+1})=a_{i+1}$.
 The assumption that $I$ is componentwise polymatroidal provides with $v=(u_{i}/x)y=x^{a_{i}-1}y^{b_{i}+1}$.
  Note that $v \neq u_{i+1}$ because $\deg (v) = \deg (u_i)> \deg(u_{i+1}) $. Therefore, $v  \notin G(I)$ and the
  only candidate in $G(I)$ that divides $v$ is $u_{i+1}$. Summarizing this, we have that $u_{i+1}$ is such that it divides
    $v=x^{a_{i}-1}y^{b_{i}+1}$, but it does not divide $u_i=x^{a_i}y^{b_i}$, and $\deg(u_i)> \deg(u_{i+1})$.
    Therefore, $u_{i+1}= x^{a_{i+1}}y^{b_{i}+1}$ with $a_i -1 >a_{i+1}$.

(ii)  and (iii)  follow from a similar argument as in (i).

To prove (iv), it is enough to show that there does not exist $1 \leq j<k< \ell \leq m$ with $d_j < d_k  > d_{\ell}$.
 Assume that such $j,k,\ell$ exist. Moreover, we may assume that for all $i$ with $j<i<\ell$ we have $d_{i}=d_k$.
  In other words, $j$ and $\ell$ are the closest integers to $k$ that satisfy  $d_j < d_k  > d_{\ell}$.  Let $k=j+t$ and $\ell=j+r$.
  Following (i),(ii) and (iii), we conclude that each $i=0, \ldots, r-1$ satisfy $a_{j+i}= a_j-i$ and $b_{\ell-i}= b_{\ell}-i$.
  Moreover,  $ a_{\ell}=a_{j+r} < a_j-r $ and $ b_{j}=b_{\ell-r} < b_{\ell}-r $.

If $d_j < d_{\ell}$, then we have $v=(u_{\ell}/y)x \in I$. The only possible element in $G(I)$ that can divide
 $v$ is $u_{\ell-1}$, but $a_{\ell-1}> a_{\ell}+1$. Hence $v \notin I$. If $d_j > d_{\ell}$,
 then we have $u=(u_j/x)y \in I$. The only possible element in $G(I)$ that can divide  $w$ is $u_{j+1}$,
  but $b_{j+1}> b_j+1$. Hence $w \notin I$. This gives us a contradiction to the assumption that $I$ satisfies non-pure dual exchange property.

  To prove the converse, take $u_i,u_k\in G(I)$ with $\deg(u_k)\geq \deg(u_i)$ and $\deg_x(u_k)< \deg_x(u_i)$.
   It is enough to show that $x(u_k/y) \in I$. It follows from $\deg_x(u_k)< \deg_x(u_i)$ that $k>i$.

    Let $k=i+1$. If $\deg(u_k)>\deg(u_{i})$, then the assertion follows from statement
 (ii). If  $\deg(u_k)=\deg(u_{i})$, then the assertion follows from statement
  (iii). Now,  let $k>i+1$. Then it follows from statement (iv) that $\deg(u_k)\geq \deg(u_{k-1})\geq \deg(u_i)$.
   Again,  if $\deg(u_k)=\deg(u_{k-1})$, then the assertion follows from statement (iii)
    and if $\deg(u_k)>\deg(u_{k-1})$, then from statement (ii) it follows that $x(u_k/y)$ is a
    multiple of $u_{k-1}$ and hence $x(u_k/y) \in I$, as required.
  \end{proof}

\begin{Corollary}\label{cortight}
    Let $I\subset K[x,y]$ be a monomial ideal with an ordering of $G(I)$ as given in $(*)$. Then $I=x^{a_m}y^{b_0} J$ is a componentwise polymatroidal
    ideal if and only if $J$ is a $yx$-tight ideal.
    \end{Corollary}
 \begin{proof}
Let $d_i= \deg (u_i)$ for all $i=0,\ldots,m$. From Theorem~\ref{prop}
(iv), it follows that there exists some $0 \leq j\leq m$ such that
\begin{equation}
d_0 \geq d_2 \geq \cdots \geq d_j \leq d_{j+1} \leq \cdots \leq d_m.
\end{equation}
Let $j=0$, that is, $d_0 \leq \cdots \leq d_m$. Then Theorem~\ref{prop} (ii) gives that $a_i-1=a_{i+1}$, for all $i=0, \ldots,m-1$.
It follows that $I=x^{a_m}y^{b_0}J$ such that $J$ is an $x$-tight ideal.

Let $j=m$, that is, $d_0 \geq \ldots \geq d_m$. Then Theorem~\ref{prop} (i) gives that $b_i+1=b_{i+1}$, for all $i=0, \ldots,m-1$.
 It follows that $I=x^{a_m}y^{b_0}J$ such that $J$ is a $y$-tight ideal.

Now let $0<j<m$, such that the sequence in (1) is neither an increasing sequence nor a decreasing sequence. Again, it follows from Theorem~\ref{prop} (i) that  $b_{i+1}=b_i+1$, for all $i=0, \ldots, j-1$ and from Theorem~\ref{prop} (ii)
 that $a_{i+1}=a_i-1$, for all $i=j,\ldots, m-1$. Hence, $I=x^{a_m}y^{b_0}J$ such that $J$ is a strict $yx$-tight ideal.

To see the converse, we first note that $I$ is componentwise polymatroidal if and only if $J$ is  componentwise polymatroidal.
 Moreover, it is a direct consequence of definition of non-pure dual exchange property that $yx$-tight
  ideals satisfy the non-pure dual exchange property. Then the conclusion holds due to Proposition~\ref{equal}.
\end{proof}

\begin{Remark}\label{complete}{\em
In simpler words, the Corollary~\ref{cortight} can be summarized as
follows: if $I\subset K[x,y]$ is a componentwise polymatroidal ideal
with $G(I)$ ordered as in $(*)$, then there exists some $0 \leq j
\leq m$ such that
\begin{enumerate}
\item[(i)] $b_0, \ldots, b_j$ is an increasing sequence of consecutive integers and $a_0, \ldots, a_j$ is a decreasing sequence of integers,
\item[(ii)] $a_j, \ldots, a_m$ is a decreasing sequence of consecutive integers and $b_j, \ldots, b_m$ is an increasing sequence of integers.
\end{enumerate}
}
\end{Remark}

Now, we prove the first main result of this section.
\begin{Theorem}\label{2quotient}
Let $I\subset S=K[x,y]$ be a componentwise polymatroidal ideal. Then
$I$ has linear quotients.
\end{Theorem}

\begin{proof}
First, we order the elements of $G(I)$ as in $(*)$. Following
Remark~\ref{complete}, we obtain some $0\leq j\leq m$ such that
$b_0, \ldots, b_j$ is an increasing sequence of consecutive integers
and $a_j, a_{j+1}, \ldots, a_m$ is a decreasing sequence of
consecutive integers. We order the elements in $G(I)$ as follows:

\begin{enumerate}
\item[{($\#$)}] \quad $u_j, \ldots, u_m, u_{j-1}, u_{j-2}, \ldots,  u_0.$
\end{enumerate}

 When $j=0$ or $j=m$, then it is clear that ($\#$) is an admissible order of $G(I)$. Now, assume that $0<j<m$. Then
 \begin{enumerate}
 \item $(u_j, \ldots, u_k): u_{k+1}= (x)$, for each $k=j, \ldots, m-1$,
 \item $(u_j, \ldots, u_m): u_{j-1}= (y)$,
 \item  $(u_j, \ldots, u_m, u_{j-1}, \ldots, u_k): u_{k-1}= (y)$ for each $k=1, \ldots, j-1$.
 \end{enumerate}
This completes the proof.
\end{proof}
The following example illustrates the admissible order described in
Theorem~\ref{2quotient}.

\begin{Example}
Let $G(I)=\{  x^{14}y^2,x^{13}y^{3},x^{10}y^4, x^9y^5, x^5y^6,
x^4y^{12},x^3y^{13}\}$. Then
\[
x^5y^6  > x^4y^{12}> x^3y^{13} > x^9y^5 > x^{10}y^4>  x^{13}y^{3}>
x^{14}y^2,
\]

is the admissible order as described in Theorem~\ref{2quotient}.
\end{Example}

The powers of componentwise polymatroidal ideals are not necessarily
componentwise polymatroidal. The ideal $I=(x_1^2, x_2^2x_3,
x_1x_2x_3, x_1x_2^2,x_1x_3^3, x_2x_3^3)\subset K[x_1,x_2,x_3]$ is
componentwise polymatroidal, however, $I^2$  fails to have this
property, as shown in \cite[Example 3.3]{BH}. Note that this ideal
$I$ is generated in the polynomial ring with three variables. We
will show that if $I,J\subset K[x,y]$ are componentwise
polaymatroidal ideals, then
 $IJ$ is again a componentwise polaymatroidal ideal. In particular,
  all powers of a componentwise polymatroidal ideal in $K[x,y]$ are again componentwise polymatroidal.
   Before stating the next result, we first recall the following result from \cite{Qu}, that will be used in Theorem~\ref{tight}.

 \begin{Proposition} \label{recall}\cite[Proposition 2.2]{Qu}
 If $I$ is $x$-tight in $[0,r]$ and $J$ is $y$-tight in $[0,s]$, then $IJ=y^sI+x^rJ$.
 \end{Proposition}

Now, we give the second main result of this section.

\begin{Theorem}\label{tight}
Let $I$ and $J$ be $yx$-tight ideals in $K[x,y]$. Then $IJ$ is also a $yx$-tight ideal.
\end{Theorem}

\begin{proof}
We divide the proof into the following four cases:

(i) If $I$ and $J$ are both either $x$-tight or $y$-tight, then the
statement follows from the details given in \cite[pg. 600]{CNJR}. In
particular, if $I$ is $x$-tight in $[0,m]$ and $J$ is $x$-tight in
$[0,n]$, then $IJ$ is $x$-tight in $[0,m+n]$.

(ii) Let $I$ be an $x$-tight ideal with $G(I)=\{x^m, x^{m-1}y^{b_1}, \ldots, y^{b_m}\}$ and $J$ be a $y$-tight
 ideal with $G(J)=\{x^{c_n}, x^{c_{n-1}}y, \ldots, y^n\}$.
 Thanks to Proposition~\ref{recall}, we have $IJ=x^mJ+y^nI$. In particular,
  \[G(x^mJ)=\{x^{c_n+m}, x^{c_{n-1}+m}y, \ldots, x^my^n\},\] and \[G(y^nI)=\{x^my^n, x^{m-1}y^{b_1+n}, \ldots, y^{b_m+n}\}.\]

From this, we conclude that $IJ$ is a $yx$-tight ideal.

\par
(iii) Let $I$ be an $x$-tight ideal with $G(I)=\{x^m, x^{m-1}y^{b_1}, \ldots, y^{b_m}\}$, and $J$ be a strict $yx$-tight ideal with
\[
G(J)=\{x^{c_n}, x^{c_{n-1}}y, \ldots,x^{n-j}y^j, x^{n-j-1}y^{d_{j+1}}, \ldots, y^{d_n}\}.
\]
Then by following the definition of $yx$-tight ideals, we can write $J=x^{n-j}L+y^{j}K$ such that $L$ is a $y$-tight ideal and $K$ is an $x$-tight ideal.
This gives $IJ=x^{n-j}IL+y^{j}IK$. Due to Proposition~\ref{recall}, we have
 \begin{eqnarray*}
IJ&=&x^{n-j}(y^jI+x^mL)+y^{j}IK\\
&=&x^{n-j}y^jI+x^{m+n-j}L+y^{j}IK\\
&=&y^j(x^{n-j}I+IK)+x^{m+n-j}L.
\end{eqnarray*}
The ideal $IK$ is $x$-tight in $[0,m+n-j]$ as discussed in part (i).
Putting $IK$ together with $x^{n-j}I$, we again get an $x$-tight
ideal in $[0,m+n-j]$. Note that $x^{m+n-j}y^j$ belongs to both $
y^j(x^{n-j}I+IK)$ and $x^{m+n-j}L$.
 Hence, $x^{m+n-j}L$ can be joined with $y^j(x^{n-j}I+IK)$, to obtain a $yx$-tight ideal, as required.

(iv)   Let $I$ be a $y$-tight ideal with $G(I)=\{x^{a_m}, x^{a_{m-1}}y, \ldots, x^{a_{1}}y^{t-1},y^t\}$, and $J$ be a strict $yx$-tight ideal with
\[
G(J)=\{x^{c_n}, x^{c_{n-1}}y, \ldots, x^{c_{n-(j-1)}}y^{j-1},x^{n-j}y^j, x^{n-j-1}y^{d_{j+1}}, \ldots, y^{d_n}\}.
\]
Then by following the definition of $yx$-tight ideals, we can write $J=x^{n-j}L+y^{j}K$ such that $L$ is a $y$-tight ideal and $K$ is an $x$-tight ideal.
This gives $IJ=x^{n-j}IL+y^{j}IK$. Due to Proposition~\ref{recall}, we have
 \begin{eqnarray*}
IJ&=&x^{n-j}IL+ y^{j} (x^{n-j}I+y^tK)\\
&=&x^{n-j}(IL+ y^{j}I)+ y^{j+t}K.
\end{eqnarray*}
 Following the similar argument as in (iii), it can be seen that
$IL+ y^{j}I$ is a $y$-tight ideal in $[0,j+t]$. Moreover, $x^{n-j}
y^{j+t}$ belongs to both $x^{n-j}(IL+ y^{j}I)$ and $ y^{j+t}K$. Hence
$IJ$ is a $yx$-tight ideal whose $x$-tight and $y$-tight parts are
joined at the monomial $x^{n-j} y^{j+t}$.

(v) Let $I$ and $J$ be strict $yx$-tight ideals with
\[
G(I)=\{x^{a_m}, x^{a_{m-1}}y, \ldots, x^{m-t}y^t, x^{m-t-1}y^{b_{t+1}}, \ldots, y^{b_m}\}
\]
and
\[
G(J)=\{x^{c_n}, x^{c_{n-1}}y, \ldots ,x^{n-j}y^j, x^{n-j-1}y^{d_{j+1}}, \ldots, y^{d_n}\}.
\]

We can write $I=x^{m-t}P+y^{t}Q$ and $J=x^{n-j}L+y^{j}K$ such that $P$ and $L$ are $y$-tight ideals, and $Q$ and $K$ are $x$-tight ideals. Then
 \begin{eqnarray*}
IJ&=&(x^{m-t}P+y^tQ)(x^{n-j}L+y^jK)\\
&=&x^{m-t+n-j}PL+x^{n-j}y^tQL+x^{m-t}y^jPK+y^{t+j}QK
\end{eqnarray*}
Using Proposition~\ref{recall}, we can write the above equality as
\[
IJ=x^{m-t+n-j}PL+x^{n-j}y^t(x^{m-t}L+y^{j}Q)+x^{m-t}y^j(x^{n-j}P+y^tK)+y^{t+j}QK
\]
Set
\[
M=x^{m-t+n-j}(PL+y^tL+y^jP)
\]
and
\[
N=y^{t+j}(x^{n-j}Q+x^{m-t}K+QK)
\]
Then $IJ=M+N$. Following the similar arguments as  in (iii) and (iv), we conclude that $PL+y^tL+y^jP$ is
 a $y$-tight ideal in $[0,t+j]$ and $x^{m-t+n-j}y^{t+j} \in M$. Also, $x^{n-j}Q+x^{m-t}K+QK$ is
 an $x$-tight ideal in $[0, m-t+n-j]$ and $x^{m-t+n-j}y^{t+j} \in N$. Hence, $IJ$ is a $yx$-tight
 ideal whose $x$-tight and $y$-tight parts are joined at the monomial $x^{m-t+n-j}y^{t+j} $.

\end{proof}

We conclude this section with the following corollary of
Theorem~\ref{tight}.

\begin{Corollary}\label{corpro}
Let $I$ and $J$ be componentwise polymatroidal ideals in $K[x,y]$. Then $IJ$ is also a componentwise polaymatroidal ideal.
\end{Corollary}

\begin{proof}
It follows immediately by Corollary~\ref{cortight}, and Theorem~\ref{tight}.
\end{proof}

\section{Componentwise polymatroidal with strong exchange property}\label{Sec3}
In this section, we give another class of componentwise
polymatroidal ideals with linear quotients. Indeed, we prove that
the componentwise polymatroidal ideals with strong exchange property
have linear quotients.

\begin{Definition}
Let $I\subset S$ be a polymatroidal ideal. We say that $I$ satisfies
the {\em strong exchange property} if  for all $u,v\in G(I)$ and for
all $i,j$ with $\deg_{x_i}(u)>\deg_{x_i}(v)$ and
$\deg_{x_j}(u)<\deg_{x_j}(v)$, one has $x_j(u/x_i)\in I$. A monomial ideal $I$ is called {\em componentwise  polymatroidal
with strong exchange property}, if each graded component of $I$ is
polymatroidal with strong exchange property.
\end{Definition}

Next, we recall the definition of ideals which are componentwise of
Veronese type.

 \begin{Definition}
Let $S=K[x_1,\ldots,x_n]$ and $d,a_1,\dots,a_n$ be positive
integers. We denote by $I_{(d;a_1,\ldots,a_n)}\subset S$, the
monomial ideal generated by all monomials $u\in S$ of degree $d$
satisfying $\deg_{x_i}(u)\leq a_i$ for all $i= 1,\ldots, n$. The
ideal $I_{(d;a_1,\ldots,a_n)}$ is called an  ideal of {\em Veronese
type}. A monomial ideal $I$ is called {\em componentwise of Veronese
type}, if each graded component of $I$ is of Veronese
type.
\end{Definition}

The ideals of Veronese type are polymatroidal and in \cite[Corollary
3.7]{BH}, it is shown that the ideals which are componentwise of
Veronese type, have linear quotients. The ideals of Veronese type
and the polymatroidal ideals with strong exchange property are
closely related. We describe their relation in the following remark
and we will use it repeatedly throughout the following text.

\begin{Remark}\label{rem}
{\em Let $I$ be a polymatroidal ideal with strong exchange property.
Then, the proof of Theorem 1.1 in \cite{HHV} shows that
$I=uI_{(d;a_1,\ldots,a_n)}$, where $u$ is the greatest common
divisor of the minimal generators of $I$.}
\end{Remark}

Remark \ref{rem} shows that the ideals of Veronese type have the
strong exchange property (set $u=1$). However, the converse of this
statement is not true. For example, let $I=x_1^2x_2I_{(4;1,3,3)}$.
Then $I$ is an ideal with
 strong exchange property, but it is not of Veronese type. To see this,
  note that if $I$ is of Veronese type, then it must be $I_{(7;3,4,3)}$.
  However, $x_2^4x_3^3\in I_{(7;3,4,3)}$ but $x_2^4x_3^3\not\in G(I)$. Therefore, it shows that  $I$ is not of Veronese type.

Our main aim in this section is to show that the componentwise
polymatroidal ideals with strong exchange property have linear
quotients.
 To achieve this aim, we will use the following concept.

\begin{Definition}
Let $I\subseteq J$ be monomial ideals with $G(I)\subseteq G(J)$. We
say that $I$ can be {\em extended by linear quotients} to $J$ if the
set $G(J)\setminus G(I)$ can be ordered $v_1,\ldots, v_m$ such that
$(G(I), v_1,\ldots, v_i): v_{i+1}$ is generated by variables for $i=
0,\ldots,m-1$. In particular, if a monomial ideal $L$ has linear
quotients, $(0)$ can be extended by linear quotients to $L$.
\end{Definition}

The following remark is a direct  consequence of above definition.

\begin{Remark}\label{extend}
\begin{enumerate}
{\em \item[(i)] Let $I\subseteq J\subseteq K$ be monomial ideals
with $G(I)\subseteq G(J)\subseteq G(K)$. If $I$ can be extended by
linear quotients to $J$ and $J$ can be extended by linear quotients
to $K$, then clearly $I$ can be extended by linear quotients to $K$.
\item[(ii)] Let $I\subseteq J $ be monomial ideals with
$G(I)\subseteq G(J)$, and $u$ be a monomial.  If $I$ can be extended
by linear quotients to $J$, then $uI$ can be extended by linear
quotients to $uJ$.}
\end{enumerate}
\end{Remark}

To prove the main result of this section, we first prove the following two results.

\begin{Lemma} \label{var}
The ideal $x_1^{c_1}\cdots x_n^{c_n}I_{(d;a_1,\ldots,a_n)}$  can be
extended by linear quotients to
$I_{(d+\sum_{i=1}^nc_i;a_1+c_1,\ldots,a_n+c_n)}$, where $c_i\geq 0$
for any $i\in[n]$.
\end{Lemma}

\begin{proof}
Following Remark~\ref{extend} (i) and (ii), it is enough to show
that for each $i\in[n]$, $K=x_iI_{(d;a_1,\ldots,a_n)}$ can be
extended by linear quotients to
$L=I_{(d+1;a_1,\ldots,a_i+1\ldots,a_n)}$. From \cite[Proposition
2.2]{BR}, we know that $L$ has linear quotients with respect to the
lexicographical order of the minimal generators induced by $x_i>x_j$ for each $j\neq i$. On the other hand, for all
monomials $u\in G(L)$ for which $x_i\mid u$, we have $u\in G(K)$.
Therefore the elements of $G(K)$ form the initial part of the admissible
order of $G(L)$, as required.
\end{proof}

\begin{Proposition}\label{veronese}
 Let $J=I_{(d;a_1,\ldots,a_n)}$ and $L=I_{(d;b_1,\ldots,b_n)}$ with $J\subseteq L$. Then $J$ can be extended by linear quotients to $L$.
\end{Proposition}

\begin{proof}
Let $K=I_{(d;a_1,\ldots, a_{s-1},a_s+1,a_{s+1},\ldots a_n)}$ for
some $s\in[n]$. Due to Remark \ref{extend} (i), it is enough to show
that $J$ can be extended by linear quotients to $K$. We order the
elements of $G(K)\setminus G(J)$ lexicographically. Moreover, for
all $u\in G(J)$ and  for all $v\in G(K)\setminus G(J)$, we set
$u>v$. We claim that, with this order, $J$ can be extended by linear
quotients to $K$. To prove the claim, we need to show that for all
$u\in G(K)$ and
 for all $v\in G(K)\setminus G(J)$, there exists a monomial  $w \in G(K)$ with
 $w>_{lex}v$  such that $(w):v=(x_r)$ and $x_r$ divides the generator of $(u):v$.

Case 1: Let $u=x_1^{h_1}\cdots x_n^{h_n}\in G(J)$ and
$v=x_1^{t_1}\cdots x_n^{t_n}\in G(K)\setminus G(J)$. Since $v\in
G(K)\setminus G(J)$, it follows that $t_s=a_s+1$, and hence
$t_s>h_s$. On the other hand, due to $\deg(u)=\deg(v)=d$, we have
that there exists some $r\in[n]$ such that $h_r>t_r$. Let $w=(v/x_s)x_r$.
Then $w\in G(J)$ and $(w):v=(x_r)$.

Case 2: Let $u,v\in G(K)\setminus G(J)$ and $u>_{lex}v$. Since $u,v\in
G(K)\setminus G(J)$, it follows that $h_s=t_s=a_s+1$. On the other
hand, due to $u>_{lex}v$, we know that there exist $l,r$ such that $r<l$,
$t_r<h_r$ and $h_l<t_l$. Let $w=(v/x_l)x_r$, then $(w):v=(x_r)$ and
$w\in G(K)\setminus G(J)$ with $w>_{lex}v$, as required.
\end{proof}

Now, we state the main result of this section.

 \begin{Theorem}\label{main}
 Let $I$ be a componentwise polymatroidal with strong exchange property. Then $I$ has linear quotients.
 \end{Theorem}

\begin{proof}
By \cite[Proposition 2.9]{SZ}, it is enough to show that $\mm
I_{\langle j \rangle}$ can be extended by linear quotients to
$I_{\langle j+1 \rangle}$, for all $j$. Let $I_{\langle j
\rangle}=uI_{(d;a_1,\ldots,a_n)}$, where $u=x_1^{u_1}\cdots
x_n^{u_n}$ is the greatest common divisor of the minimal generators
of $I_{\langle j \rangle}$ and let $I_{\langle j+1
\rangle}=vI_{(l;b_1,\ldots,b_n)}$, where $v=x_1^{v_1}\cdots
x_n^{v_n}$ is the greatest common divisor of the minimal generators
of $I_{\langle j+1 \rangle}$. Since $\mm I_{\langle j \rangle}
\subseteq I_{\langle j+1 \rangle}$, it follows that $G(\mm
I_{\langle j \rangle})\subseteq G(I_{\langle j+1 \rangle})$. Fix an
integer $i\in [n]$. Since  $u$ is the greatest common divisor of the
minimal generators of $I_{\langle j \rangle}$,  there exists $w\in
G(I_{(d;a_1,\ldots,a_n)})$ such that $x_i\nmid w$. Since $uwx_r\in
G(\mm I_{\langle j \rangle})$, with $r\neq i$, it follows that there
exists some $w'\in G(I_{(l;b_1,\ldots,b_n)})$ such that $uwx_r=vw'$.
Hence $u_i=\deg_{x_i}(uwx_r)=\deg_{x_i}(vw')\geq v_i$, and hence
$u_i\geq v_i$. We may assume that $u_1=v_1+c_1,\ldots,u_s=v_s+c_s$,
and $u_{s+1}=v_{s+1},\ldots,u_n=v_n$
 for some $s\in[n]$.
By \cite[Theorem 3.6]{BH} and Remark \ref{extend} (ii),
$\mm I_{\langle j \rangle}=u\mm I_{(d;a_1,\ldots,a_n)}$ can be extended by linear quotients to
$uI_{(d+1;a_1+1,\ldots,a_n+1)}$. On the other
hand, using Lemma \ref{var}, we have that
$$uI_{(d+1;a_1+1,\ldots,a_n+1)}=vx_1^{c_1}\cdots
x_s^{c_s}I_{(d+1;a_1+1,\ldots,a_n+1)}$$
 can be extended by linear quotients to
$$vI_{(d+\sum_{i=1}^sc_i+1;a_1+c_1+1,\ldots,a_s+c_s+1,a_{s+1}+1,\ldots,a_n+1)}.$$ Since $\mm I_{\langle j \rangle} \subseteq I_{\langle j+1 \rangle}$, it
follows that
$$a_1+c_1+1\leq b_1,\ldots,a_s+c_s+1\leq b_s\;,\;a_{s+1}+1\leq b_{s+1},\ldots,a_n+1\leq b_n.$$
Also, we have
$$d+\sum_{i=1}^sc_i+1=d+\sum_{i=1}^nu_i-\sum_{i=1}^nv_i+1=j-\sum_{i=1}^nv_i+1=l.$$
Then, it follows from Proposition \ref{veronese} and Remark
\ref{extend} (ii) that
$$vI_{(d+\sum_{i=1}^sc_i+1;a_1+c_1+1,\ldots,a_s+c_s+1,a_{s+1}+1,\ldots,a_n+1)}$$
 can be extended by linear quotients to $vI_{(l;b_1,\ldots,b_n)}=I_{\langle j+1 \rangle}$. Therefore, the desired result follows by Remark \ref{extend} (i).
\end{proof}

We conclude this section with the following notes. If $I$ is a
componentwise polymatroidal ideal with strong exchange property,
then powers of $I$ need not to be componentwise polymatroidal. The
ideal $I=(x_1^2, x_2^2x_3, x_1x_2x_3, x_1x_2^2, x_1x_3^3, x_2x_3^3)$
is componentwise polymatroidal with strong exchange property, while
$I^2$ is not componentwise polymatroidal, see \cite[Example
3.3]{BH}. In fact, in a very special case, when we take the product
of the maximal ideal with an ideal of Veronese type,
 the resulting ideal need not to have strong exchange property. For
 example, consider
 the ideal $I=\mm I_{(6;3,2,1,4)}$. The maximal ideal $\mm$ and  $ I_{(6;3,2,1,4)}$ admit the strong exchange property. However, $I$
 does not have strong exchange property. To see this, note that $x_1^4x_2^2x_3,x_2^3x_4^4\in G(I)$, but $x_1^4x_2^3\notin G(I)$.

\end{document}